\documentclass[a4paper]{article}
\usepackage{amsmath,amsthm,amsfonts}
\usepackage{amssymb}
\usepackage[T1]{fontenc}

\newcommand{\ip}[2]{\bigl\langle #1,\, #2\bigr\rangle}	% Pairing of measure and function
\newcommand{\cm}{\pmb{\nu}}

\usepackage{graphicx}
\usepackage{subfigure}
\usepackage{euscript}
\usepackage{units}
\usepackage{enumitem}
\usepackage[bookmarks=true,hidelinks]{hyperref}
\usepackage{bbm}
\usepackage{bm}
\usepackage{xcolor}

\newtheorem{theorem}{Theorem}[section]
\newtheorem{lemma}[theorem]{Lemma}
\newtheorem{proposition}[theorem]{Proposition}

\theoremstyle{definition}
\newtheorem{definition}[theorem]{Definition}
\newtheorem{remark}[theorem]{Remark}

\numberwithin{equation}{section}

\def\Xint#1{\mathchoice
{\XXint\displaystyle\textstyle{#1}} 
{\XXint\textstyle\scriptstyle{#1}} 
{\XXint\scriptstyle\scriptscriptstyle{#1}} 
{\XXint\scriptscriptstyle\scriptscriptstyle{#1}} 
\!\int}
\def\XXint#1#2#3{{\setbox0=\hbox{$#1{#2#3}{\int}$ }
\vcenter{\hbox{$#2#3$ }}\kern-.56\wd0}}
\def\intavg{\Xint-}

\renewcommand{\geq}{\geqslant}
\renewcommand{\leq}{\leqslant}
\newcommand{\define}[1]{\emph{#1}}
\newcommand{\Prob}{\EuScript{P}}

\renewcommand{\epsilon}{\varepsilon}
\newcommand{\eps} {\varepsilon}
\renewcommand{\phi}{\varphi}
\newcommand{\R}{\mathbb{R}}
\newcommand{\N}{\mathbb{N}}

\newcommand{\Lp}{\EuScript{L}}

\newcommand{\phase}{\mathbf{U}}
\newcommand{\timedom}{{\mathcal{T}}}	% The time domain

\begin{document}

\title{Statistical solutions and Onsager's conjecture}
\author{U. S. Fjordholm\thanks{Department of Mathematical Sciences, Norwegian University of Science and Technology, Trondheim, N-7491, Norway.} \and E. Wiedemann\thanks{
Institut f\"ur Angewandte Mathematik, Leibniz Universit\"at Hannover, Welfengarten 1, 30167 Hannover, Germany.}}

\maketitle
\begin{abstract}
We prove a version of Onsager's conjecture on the conservation of energy for the incompressible Euler equations in the context of statistical solutions, as introduced recently by Fjordholm et al. As a byproduct, we also obtain a new proof for the conservative direction of Onsager's conjecture for weak solutions.
\end{abstract}

\begin{center}
{\textit{Dedicated to Edriss S. Titi on the occasion of his 60th birthday.}}
\end{center}

\section{Introduction}
We consider the $d$-dimensional incompressible Euler equations: Find a function $v = (v^1,\dots,v^d) : \R_+\times D \to \R^d$ and a function $p : \R_+\times D\to\R$ such that
\begin{equation}\label{eq:euler}
\begin{split}
\partial_t v + \sum_k\partial_{x^k} (vv^k) + \nabla p = 0 &\qquad x\in D,\ t>0 \\
\nabla\cdot v = 0 &\qquad x\in D,\ t>0 \\
v(0,x) = v_0(x) &\qquad x\in D.
\end{split}
\end{equation}
Here and below, the summation limits, when not specified, are always from $k=1$ to $k=d$. The initial data $v_0$ is assumed to lie in $L^2(D)$. The spatial parameter $x$ takes values in a set $D$, which we will take as either $\R^d$ or the ($d$-dimensional) torus $\mathbb{T}^d$\footnote{On domains with boundaries, the situation is more subtle: While one can show the \emph{local} version of the energy equality with almost no further effort, it is no longer clear whether this implies also the \emph{global} conservation of energy under the usual Besov regularity assumption.}. The temporal domain is $[0,T]$ for some $T>0$. By a \textit{solution of the Euler equations} we will mean a weak solution of \eqref{eq:euler}, i.e.\ a function $v \in L^2_{\textit{loc}}\big([0,T]\times D;\R^d\big)$ such that
\begin{equation}\label{eq:weakdef}
\int_{\R_+}\int_D v\partial_t\phi + \sum_k vv^k \partial_{x^k} \phi + p\nabla \phi\,dxdt + \int_D v_0(x)\phi(0,x)\,dx = 0
\end{equation}
for all $\phi \in C_c^\infty([0,T)\times\R^d)$, as well as satisfying the divergence free condition in the sense of distributions. 

Assume that $v$ is a smooth solution of \eqref{eq:euler}. Multiplying the first equation of \eqref{eq:euler} by $v$ gives
\begin{equation}\label{eq:localenergycons}
\partial_t \frac{|v|^2}{2} + \sum_k \partial_{x^k} \left(v_k\left(\frac{|v|^2}{2} + p\right)\right) = 0.
\end{equation}
Integrating this \textit{local energy identity} over $x\in D$ and $t\in[0,T]$, we obtain the \textit{global energy identity}
\begin{equation}\label{eq:globalenergycons}
\int_D\frac{|v(T)|^2}{2}\,dx = \int_D \frac{|v_0|^2}{2}\,dx.
\end{equation}
In 1949 Lars Onsager conjectured \cite{Ons49} that if $v$ is H\"older continuous with exponent greater than $1/3$ then the above calculations can be made rigorous:
\begin{quote}
In fact it is possible to show that the velocity field in such ``ideal'' turbulence cannot obey any \textsc{Lipschitz} condition of the form (...) for any order $n$ greater than $1/3$; otherwise the energy is conserved.\footnote{As a historical sidenote, Onsager wrote down a formal proof of his own conjecture---close in spirit to the later proof by Duchon and Robert---which he never published; see \cite{EyiSre06}. Thus, his rather cryptic ``it is possible to show'' should in fact be interpreted literally rather than hypothetically.}
\end{quote}
As was shown in 1994 by Constantin, E and Titi \cite{CET94}, the conjecture {is indeed} true when $\frac{1}{3}$-H\"older continuity is replaced by $B_3^{\alpha,\infty}$ (Besov) regularity for any $\alpha>\frac{1}{3}$. The proof uses a regularization of \eqref{eq:euler} together with some basic estimates in Besov spaces. Independently, Eyink \cite{Eyi94} proved the conjecture in Fourier space under a stronger assumption. Duchon and Robert \cite{DR00} employed the regularization technique from \cite{CET94} to quantify the \emph{anomalous energy dissipation} $\mathcal{E}(u)$ of an arbitrary solution $u$---the amount by which equality in \eqref{eq:localenergycons} fails to hold. The sharp exponent $\alpha=\frac{1}{3}$ was shown to suffice for energy conservation at the cost of a slightly stronger summability assumption for the Besov space in \cite{CCFS08}, see also \cite{RRS16}. More recently, related results were given for density-dependent Euler models in \cite{LS16, FGSW17, DE17}. On the other hand, there is the question whether energy \emph{can} be dissipated for any H\"older or Besov regularity below the exponent $\frac{1}{3}$. This difficult problem was solved only very recently \cite{Is17, BDSV17}.

In \cite{FLM17} the authors developed the concept of \emph{statistical solutions} of hyperbolic conservation laws, which are solutions of an evolution equation which incorporate uncertainty in the solution. The uncertainty can be due to errors in the initial or boundary data or---as is more relevant in the present setting---modeling in the context of turbulent flows. Statistical solutions were formulated as maps from time $t$ to probability measures $\mu_t$ on $L^2(D)$. The authors showed that any probability measure $\mu\in\Prob(L^2(D))$ can be described equivalently as a \emph{correlation measure}, a hierarchy $\cm=(\nu^1,\nu^2,\dots)$ in which each element $\nu^k$ gives the joint probability distribution $\nu^k_{x_1,\dots,x_k}$ of the values $v(x_1),\dots,v(x_k)$ for any choice of spatial points $x_1,\dots,x_k\in D$. The evolution equation for $\mu_t$ is most naturally described in terms of its corresponding correlation measure, yielding an infinite hierarchy of evolution equations. In particular, the equation for the one-point distribution $\nu^1_x$ coincides with {DiPerna's definition }of \emph{measure-valued solutions}, see~\cite{DiP85,DM87,SW12}.

The purpose of the present paper is twofold. First, in Section \ref{sec:onsagerweaksoln} we provide an alternative, rather elementary proof of (local) energy conservation of solutions of the Euler equation which is close in spirit to the regularization technique in \cite{DR00}, but is also reminiscent of Kruzhkov's \emph{doubling of variables} technique \cite{Kr70}. Second, in Section \ref{sec:statsoln} we formulate the concept of statistical solutions of the incompressible Euler equations \eqref{eq:euler}, and we show that the proof of energy conservation is readily generalized to statistical solutions. {We end by comparing the concepts of Besov regularity of functions and of correlation measures in Section \ref{sec:regularity}.}

Our main result (Theorem \ref{thm:statsolnenergycons}) says that an uncertain fluid flow---realized as a statistical solution of \eqref{eq:euler}---conserves energy provided it has more than $\frac{1}{3}$ of a derivative. The ease by which statements about regularity can be formulated using correlation measures indicates to us that statistical solutions---and not measure-valued solutions---are the right notion of solutions for uncertain (or \emph{unsteady}) fluid flows. {We refer the interested reader to the upcoming paper \cite{FMW17} where we discuss statistical solutions of the Navier--Stokes and Euler equations and the connection to Kolmogorov's theory of turbulence (cf.~also Remark \ref{rem:kolmogorov}).}

\section{Onsager's conjecture}\label{sec:onsagerweaksoln}
We first write the Euler equation in component form:
\begin{equation}\label{eq:eulercomp}	% Euler in component form
\partial_t v^i + \sum_k\partial_{x^k} (v^iv^k) + \partial_{x^i}p = 0.
\end{equation}
It is straightforward to see that if $(v,p)$ is any smooth solution of the above equation, then the function $(t,x,y) \mapsto v^i(t,x)v^j(t,y)$ satisfies
\begin{equation}\label{eq:euler2ndmoment}
\begin{split}
\partial_t\left(v^i(x)v^j(y)\right) + \sum_k \partial_{x^k}\left(v^i(x)v^k(x)v^j(y)\right) + \sum_k\partial_{y^k}\left (v^i(x)v^k(y)v^j(y)\right) \\
~ + \partial_{x^i}\left(p(x)v^j(y)\right) + \partial_{y^j} \left(v^i(x)p(y)\right) = 0
\end{split}
\end{equation}
(where we suppress the dependence on $t$). The proof of this claim consists of evaluating \eqref{eq:eulercomp} at $x$ and at $y$, multiplying the former by $v^j(y)$ and the latter by $v^i(x)$, and then summing the two. With only a bit more work, one shows that if $v$ is any \textit{weak} solution of the Euler equation then \eqref{eq:euler2ndmoment} is satisfied in the distributional sense.

\begin{lemma}
If $v : (0,T)\times D\to\R^d$ is a weak solution of the Euler equation then
\begin{equation}\label{eq:weakeuler2ndmoment}
\begin{split}
\int_{0}^T\int_D\int_Dv^i(x)v^j(y)\partial_t\phi + \sum_k v^i(x)v^k(x)v^j(y)\partial_{x^k}\phi + \sum_k v^i(x)v^k(y)v^j(y) \partial_{y^k} \phi \\
~ + p(x)v^j(y)\partial_{x^i}\phi + v^i(x)p(y) \partial_{y^j} \phi\,dxdydt + \int_D\int_Dv_0^i(x)v_0^j(y)\phi(0,x,y)\,dxdy = 0
\end{split}
\end{equation}
for every test function $\phi = \phi(t,x,y)\in C_c^\infty([0,T)\times D^2)$ and every $i,j=1,\dots,d$.
\end{lemma}

If the local energy inequality \eqref{eq:localenergycons} is to hold for a weak solution, then we need in addition $v\in L^3_{loc}\big([0,T)\times D;\R^d\big)$ and $p\in L_{loc}^{3/2}\big([0,T)\times D;\R^d\big)$ for the equality to make sense distributionally. The Besov regularity of $v$ required to show the equality entails in particular $v\in L^3$, and this in turn implies $p\in L^{3/2}$. Indeed, taking the divergence of the momentum equation in \eqref{eq:euler} we obtain
\begin{equation}
-\Delta p=\sum_{k,l}\partial^2_{x^k,x^l}(v^lv^k),
\end{equation}   
and thus $v\in L^3$ implies $p\in L^{3/2}$ by standard elliptic theory.

\begin{theorem}
Let $v \in L^\infty\big((0,T); L^2\big(D;\R^d\big)\big)\cap L^3((0,T)\times D;\R^d)$ be a weak solution of the Euler equations, and accordingly $p\in  L^{3/2}((0,T)\times D;\R^d)$. Assume that
\begin{equation}\label{eq:besovregular}
\liminf_{\eps\to0}\frac{1}{\eps}\int_0^T\int_D\intavg_{B_\eps(x)} \big|v(x)-v(y)\big|^3\,dydxdt = 0.
\end{equation}
Then $(v, p)$ satisfies the local energy identity \eqref{eq:localenergycons}.
\end{theorem}

\begin{proof}
\newcommand{\term}{\mathcal{A}}
Set $j=i$ in \eqref{eq:weakeuler2ndmoment} and sum over $i$. We can then write the resulting identity as
\begin{equation}\label{eq:weakeuler2ndmoment_2}
\begin{split}
\int_{0}^T\int_D\int_Dv(x)\cdot v(y)\partial_t\phi + v(x)\cdot v(y)\left(v(x)\cdot \nabla_x\phi + v(y)\cdot \nabla_y \phi\right) + p(x)v(y)\cdot\nabla_x\phi \\ + p(y)v(x)\cdot\nabla_y\phi\,dxdydt + \int_D\int_Dv_0(x)\cdot v_0(y)\phi(0,x,y)\,dxdy = 0.
\end{split}
\end{equation}
Fix a number $\eps>0$. We choose now the test function $\phi(t,x,y) = \rho_\eps(x-y)\psi\left(t,x\right)$, where  $\rho_\eps(z) = \eps^{-d}\rho(\eps^{-1}z)$ for a nonnegative, rotationally symmetric mollifier $\rho\in C_c^\infty(D)$ with unit mass and support in $B_0(1)$ (the unit ball in $\R^d$) and $\psi\in C_c^\infty((0,T)\times D)$. Then
\[
\partial_t\phi = \rho_\eps\partial_t\psi, \qquad 
\nabla_x \phi = \psi\nabla\rho_\eps+\rho_\eps\nabla\psi, \qquad 
\nabla_y \phi = -\psi\nabla\rho_\eps.
\]
Continuing from \eqref{eq:weakeuler2ndmoment_2}, we now have
\begin{equation}\label{eq:weakeuler2ndmoment_3}
\begin{split}
\int_{0}^T\int_D\int_Dv(x)\cdot v(y)\rho_\eps\partial_t\psi + v(x)\cdot v(y)\bigl(v(x) - v(y)\bigr)\cdot\nabla\rho_\eps\psi \\
+  v(x)\cdot v(y) v(x)\cdot\nabla\psi\rho_\eps + \big(p(x)v(y)- v(x)p(y)\big)\cdot\nabla\rho_\eps\psi\\
+ p(x)v(y)\cdot\nabla\psi\rho_\eps\,dxdydt 
%\\ + \int_D\int_D v_0(x)\cdot v_0(y)\rho_\eps\psi(0,x)\,dxdy= 0.
\end{split}
\end{equation}
Making the change of variables $z = x-y$ gives
\begin{equation}\label{eq:weakeuler2ndmoment_4}
\begin{split}
\int_{0}^T  \int_D\int_D\partial_t\psi(t,x)\rho_\eps(z)v(x)\cdot v(x-z)\,dxdzdt \\
+ \int_{0}^T  \int_D\int_D\psi(t,x)\nabla\rho_\eps(z)\cdot\bigl(v(x) - v(x-z)\bigr) \bigl(v(x)\cdot v(x-z)\bigr)\,dxdzdt \\
+ \int_{0}^T  \int_D\int_D\nabla\psi(t,x)\rho_\eps(z)\cdot v(x) \bigl(v(x)\cdot v(x-z)\bigr)\,dxdzdt \\
+ \int_{0}^T\int_D\int_D\psi(t,x)\nabla\rho_\eps(z)\cdot (p(x)v(x-z)- v(x)p(x-z))\,dxdzdt \\
+ \int_{0}^T\int_D\int_D\nabla\psi(t,x)\rho_\eps(z)\cdot v(x-z)p(x)\,dxdzdt = 0.
\end{split}
\end{equation}
Decompose the above into a sum of five terms $\term_1+\dots+\term_5$. By applying the Lebesgue differentiation theorem, it is easy to see that
\begin{align*}
\term_1 &\to \int_{0}^T\int_D |v(t,x)|^2 \partial_t \psi(t,x)\,dxdt, \\
\term_3+\term_5 &\to \int_{0}^T\int_D \big(|v(t,x)|^2+p(t,x)\big)v(t,x)\cdot\nabla \psi(t,x)\,dxdt
\end{align*}
as $\eps\to0$. For $\term_2$ we can write $\term_2 = \term_{2,1} + \term_{2,2}$, where
\begin{align*}
\term_{2,1} &= -\frac{1}{2}\int_{0}^T \int_D\int_D\psi(t,x)\nabla\rho_\eps(z)\cdot\bigl(v(x) - v(x-z)\bigr) \bigl|v(x)- v(x-z)\bigr|^2\,dxdzdt, \\
\term_{2,2} &= \frac{1}{2}\int_{0}^T \int_D\int_D\psi(t,x)\nabla\rho_\eps(z)\cdot\bigl(v(x) - v(x-z)\bigr) \left(|v(x)|^2+ |v(x-z)|^2\right)\,dxdzdt.
\end{align*}
The first term can be bounded by
\begin{align*}
\frac{\|\psi\|_\infty}{2} \int_0^T\int_D\int_D |\nabla\rho_\eps(z)|\big|v(x)-v(x-z)\big|^3\ dzdxdt \\
\leq C\frac{1}{\eps} \int_0^T\int_D\intavg_{B_\eps(0)}\big|v(x)-v(x-z)\big|^3\ dzdxdt
\end{align*}
the inequality following from the fact that $\|\nabla\rho_\eps\|_{L^1(D)} \leq C\eps^{-1}$. By the Besov regularity assumption \eqref{eq:besovregular}, the above vanishes along a subsequence $\eps'\to0$. In the second term $\term_{2,2}$ we make the change of variables $x \mapsto x+z$ in the term $v(x-z)$ and then the change of variables $z \mapsto -z$ to obtain
\[
\term_{2,2} = \frac{1}{2}\int_{0}^T \int_D\int_D\nabla\rho_\eps(z)\cdot|v(x)|^2\bigl(v(x) - v(x-z)\bigr)\big(\psi(x)+\psi(x-z)\big)\,dxdzdt.
\]
Writing $\psi\nabla\rho_\eps= \nabla(\psi\rho_\eps)-\rho_\eps\nabla\psi$ and using the divergence constraint, we obtain
\[
\term_{2,2} = \frac{1}{2}\int_{0}^T \int_D\int_D\rho_\eps(z)|v(x)|^2\nabla\psi(t,x-z)\cdot\bigl(v(x) - v(x-z)\bigr)\,dxdzdt.
\]
The Lebesgue differentiation theorem now implies that $\term_{2,2}\to0$ as $\eps\to0$.

For $\term_4$, the change of variables $x\mapsto x+z$ in the second term gives
\[
\term_4 = \int_{0}^T\int_D\int_D p(x)\nabla\rho_\eps(z)\cdot \big(\psi(x)v(x-z)-\psi(x+z)v(x+z)\big)\,dxdzdt.
\]
The divergence constraint implies that the first term is zero, while the second term gives
\[
\term_4 = \int_{0}^T\int_D\int_D p(x)\rho_\eps(z)\nabla\psi(x+z)\cdot v(x+z)\big)\,dxdzdt
\]
which converges to $\int_{0}^T\int_D p(x)\nabla\psi\cdot v\,dxdt$. Summing up all the terms, we conclude that in the limit $\eps\to0$ we obtain the local energy identity \eqref{eq:localenergycons} in distributional form.
\end{proof}

\section{Statistical solutions of the Euler equations}\label{sec:statsoln}
In this section we prove that statistical solutions of Euler's equation are energy conservative under a Besov-type regularity assumption. In Section \ref{sec:corrmeas} we introduce the necessary technical machinery, in Section \ref{sec:statsolndef} we define statistical solutions of Euler's equation and in Section \ref{sec:onsagerstatsoln} we carry out the proof of energy conservation. The proof closely follows the proof of energy conservation for weak solutions in Section \ref{sec:onsagerweaksoln}. In particular, there are direct analogues of the weak formulation(s), the divergence constraint, the Besov regularity assumption, and the Lebesgue differentiation theorem.

\subsection{Correlation measures}\label{sec:corrmeas}
\begin{definition}
Let $d,N\in\N$, $q\in[1,\infty)$ and let $D\subset\R^d$ be an open set (the ``space domain'') and denote $\phase = \R^N$ (``phase space''). A \textit{correlation measure from $D$ to $\phase$} is a collection $\cm = (\nu^1,\nu^2,\dots)$ of maps satisfying:
\begin{enumerate}[label=\it(\roman*)]
\item $\nu^k$ is a Young measure from $D^k$ to $\phase^k$.
\item \textit{Symmetry:} if $\sigma$ is a permutation of $\{1,\dots,k\}$ and $f\in C_0(\phase^k)$ then $\ip{\nu^k_{\sigma(x)}}{f(\sigma(\xi))} = \ip{\nu^k_{x}}{f(\xi)}$ for a.e.\ $x\in D^k$. %Here, we denote $\sigma(x) = \sigma(x_1,x_2,\ldots, x_k) = (x_{\sigma_1},x_{\sigma_2},\ldots,x_{\sigma_l})$, and $\sigma(\xi)$ is denoted analogously.
\item \textit{Consistency:} If $f\in C_b(\phase^k)$ is of the form $f(\xi_1,\dots,\xi_k) = g(\xi_1,\dots,\xi_{k-1})$ for some $g\in C_0(\phase^{k-1})$, then $\ip{\nu^k_{x_1,\dots,x_k}}{f} = \ip{\nu^{k-1}_{x_1,\dots,x_{k-1}}}{g}$ for almost every $(x_1,\dots,x_k)\in D^k$.
\item \textit{$L^q$ integrability:} 
\begin{equation}\label{eq:corrlqbound_stationary}
\int_D \ip{\nu^1_{x}}{|\xi|^q}\,dx < \infty.
\end{equation}
\item\textit{Diagonal continuity (DC):} $\lim_{\eps\to0}d_\eps^q(\nu^2) = 0$, where
\begin{equation}\label{eq:defd_eps}
d_\eps^q(\nu^2) := \left(\int_D \intavg_{B_\eps(x)} \ip{\nu^2_{x,y}}{|\xi_1-\xi_2|^q}\,dydx\right)^{1/q}.
\end{equation}
\end{enumerate}
We denote the set of all correlation measures by $\Lp^q(D,\phase)$.
\end{definition}

\begin{remark}\label{rem:depsbounded}
The ``modulus of continuity'' $d_\eps^q(\nu^2)$ is bounded irrespective of $\eps>0$, due to the $L^q$ bound. Indeed,
\begin{align*}
%d_\eps^q(\nu^2) &= \int_D \intavg_{B_\eps(x)} \ip{\nu^2_{x,y}}{|\xi_1-\xi_2|^q}\,dydx \\
%&\leq 2^{q-1}\int_D \intavg_{B_\eps(x)} \ip{\nu^2_{x,y}}{|\xi_1|^q} + \ip{\nu^2_{x,y}}{|\xi_2|^q}\,dydx \\
%&= 2^q\int_D \ip{\nu^1_x}{|\xi|^q}\,dx < \infty,
d_\eps^q(\nu^2) &= \left(\int_D \intavg_{B_\eps(x)} \ip{\nu^2_{x,y}}{|\xi_1-\xi_2|^q}\,dydx\right)^{1/q} \\
&\leq \left(\int_D \intavg_{B_\eps(x)} \ip{\nu^2_{x,y}}{|\xi_1|^q}\,dydx\right)^{1/q} + \left(\int_D \intavg_{B_\eps(x)}\ip{\nu^2_{x,y}}{|\xi_2|^q}\,dydx\right)^{1/q} \\
&= 2\left(\int_D \ip{\nu^1_x}{|\xi|^q}\,dx\right)^{1/q} < \infty,
\end{align*}
where we have used Minkowski's inequality and then the consistency requirement.
\end{remark}

\begin{remark}
An example of a correlation measure is $\nu^k_{x_1,\dots,x_k} = \delta_{u(x_1)}\otimes\cdots\otimes\delta_{u(x_k)}$, where $x_1,\dots,x_k\in D$, $k\in\N$, $u\in L^q(D,\R^N)$ is a measurable function and $\delta_v$ is the Dirac measure centered at $v\in\R^N$. The symmetry and consistency conditions \textit{(ii)}, \textit{(iii)} follow immediately, and the $L^q$ integrability condition \textit{(iv)} asserts that $u \in L^q(D,\R^N)$. Moreover, the modulus of continuity $d_\eps^q$ in condition \textit{(v)} is
\[
d_\eps^q(\nu^2) = \left(\int_D \intavg_{B_\eps(x)} |u(x)-u(y)|^q\,dydx\right)^{1/q},
\]
which vanishes as $\eps\to0$ due to the Lebesgue differentiation theorem. Thus, diagonal continuity is the assertion that the Lebesgue differentiation theorem---which automatically holds for $L^q$ functions---also holds for $\cm$. Correlation measures which are concentrated on a single function $u$ are called \textit{atomic}.
\end{remark}

\begin{definition}\label{def:corrmeas}
Let $d,N\in\N$ and let $D\subset\R^d$ be an open set (the ``space domain''), let $\timedom\subset\R$ be an interval (the ``time domain'') and denote $\phase = \R^N$ (``phase space''). A \textit{time-dependent correlation measure from $\timedom\times D$ to $\phase$} is a collection $\cm = (\nu^1,\nu^2,\dots)$ of maps satisfying:
\begin{enumerate}[label=\it(\roman*)]
\item $\nu^k$ is a Young measure from $\timedom\times D^k$ to $\phase^k$.
\item \textit{Symmetry:} if $\sigma$ is a permutation of $\{1,\dots,k\}$ and $f\in C_0(\phase^k)$ then $\ip{\nu^k_{t,\sigma(x)}}{f(\sigma(\xi))} = \ip{\nu^k_{t,x}}{f(\xi)}$ for a.e.\ $(t,x)\in \timedom\times D^k$. %Here, we denote $\sigma(x) = \sigma(x_1,x_2,\ldots, x_k) = (x_{\sigma_1},x_{\sigma_2},\ldots,x_{\sigma_l})$, and $\sigma(\xi)$ is denoted analogously.
\item \textit{Consistency:} If $f\in C_b(\phase^k)$ is of the form $f(\xi_1,\dots,\xi_k) = g(\xi_1,\dots,\xi_{k-1})$ for some $g\in C_0(\phase^{k-1})$, then $\ip{\nu^k_{t,x_1,\dots,x_k}}{f} = \ip{\nu^{k-1}_{t,x_1,\dots,x_{k-1}}}{g}$ for almost every $(t,x_1,\dots,x_k)\in \timedom\times D^k$.
\item \textit{$L^q$ integrability:} There is a $c>0$ such that
\begin{equation}\label{eq:corrlpbound}
\int_D \ip{\nu^1_{t,x}}{|v|^q}\,dx \leq c \qquad \text{for a.e.\ } t\in\timedom.
\end{equation}
\item\textit{Diagonal continuity (DC):} $\lim_{\eps\to0}d_\eps^q(\nu^2) = 0$, where \begin{equation}\label{eq:defd_epstime}
d_\eps^q(\nu^2) := \left(\int_0^T\int_D \intavg_{B_\eps(x)} \ip{\nu^2_{t,x,y}}{|\xi_1-\xi_2|^q}\,dydxdt\right)^{1/q}.
\end{equation}
\end{enumerate}
\end{definition}

\subsection{Statistical solutions}\label{sec:statsolndef}
For the following definition, recall that the natural framework to study the local energy (in)equality for the incompressible Euler equations is $v\in L^3_{t,x}$, $p\in L^{3/2}_{t,x}$. Therefore, we need to distinguish the integrability conditions corresponding to the velocity and the pressure, respectively, which leads to the ``mixed" integrability and diagonal continuity conditions \eqref{eq:L3} and \eqref{eq:DCmix} below. The integration variable of $\nu^k_{t,x}$ will be denoted $\xi = (v^1,\dots,v^d,q) \in \phase^k$, with the interpretation of $v^i=(v^i_1,\dots,v^i_k)$ as the $i$-th component of velocity and of $p=(p_1,\dots,p_k)$ as the scalar pressure at $k$ different spatial points $x=(x_1,\dots,x_k)$. Here, our phase space is $\phase=\R^{d+1}$.

\begin{definition}
Let $D\subset\R^d$ be a spatial domain and let $T>0$. Let $\bar{\cm}\in\Lp^2(D,\R^d)$ be given initial data. By a \textit{statistical solution} of the incompressible Euler equations, we will mean a time-dependent correlation measure $\cm$ from $[0,T]\times D$ to $\R^{d+1}$, where the integrability condition \eqref{eq:corrlpbound} is to be understood as 
\begin{equation}\label{eq:L3}
\begin{cases}
\displaystyle\int_D\ip{\nu^1_{t,x}}{|v|^2}\,dx \leq c & \text{for a.e.\ } t\in[0,T] \\
\displaystyle\int_0^T\int_D\ip{\nu^1_{t,x}}{|v|^3+|p|^{3/2}}\,dxdt<\infty
\end{cases}
\end{equation}\label{eq:DCmix}
for some $c>0$, and diagonal continuity is to be understood as
\begin{equation}
\lim_{\eps\to0}\int_0^T\int_D \intavg_{B_\eps(x)} \ip{\nu^2_{t,x,y}}{|v_1-v_2|^2+|v_1-v_2|^3+|p_1-p_2|^{3/2}}\,dydxdt=0,
\end{equation}
such that:
\begin{enumerate}[label=\it(\roman*)]
\item For all $k\in\N$, $\nu^k = \nu^k_{t,x_1,\dots,x_k}$ satisfies
\begin{equation}\label{eq:MVeulerk}
\begin{split}
\int_0^T\int_{D^k}\ip{\nu^k}{v^{i_1}_1\cdots v^{i_k}_k}\partial_t\phi + \sum_{l=1}^k\ip{\nu^k}{v_l v^{i_1}_1\cdots v^{i_k}_k}\cdot \nabla_{x_l}\phi \\
+ \sum_{l=1}^k\ip{\nu^k}{v_1^{i_1}\cdots p_l \cdots v_k^{i_k}}\frac{\partial\phi}{\partial x_l^{i_l}}\,dxdt  = 0
\end{split}
\end{equation} 
for all $i_1,\dots,i_k=1,\dots,d$ and for all $\phi\in C_c^\infty\big((0,T)\times D^k\big)$. (Here we abbreviate $v_1^{i_1}\cdots p_l \cdots v_k^{i_k} = v_1^{i_1}\cdots v_{l-1}^{i_{l-1}} p_l v_{l+1}^{i_{l+1}} \cdots v_k^{i_k}$, the $l$th component of $v$ being omitted.)
\item $\cm$ is divergence-free: For every $\phi\in C_c^\infty(\R_+\times D^k)$ and every $\kappa\in C(\phase^{k-1})$ for which $\ip{\nu^{k-1}}{|\kappa|} < \infty$, we have
\begin{equation}\label{eq:divfree}
\int_{0}^T\int_{D^k}\nabla_{x_k}\phi(x_k)\cdot \ip{\nu^k_{t,x}}{\kappa(v_1,\dots,v_{k-1}) v_k}\,dxdt = 0.
\end{equation}
\end{enumerate}
\end{definition}

\begin{remark}
The first two instances of \eqref{eq:MVeulerk} are:
\begin{equation}\label{eq:aveuler}
\begin{split}
\int_0^T\int_D \ip{\nu^1}{v^i}\partial_t\phi + \ip{\nu^1}{v^iv}\cdot\nabla\phi + \ip{\nu^1}{p}\partial_{x^i}\phi\,dxdt  = 0
\end{split}
\end{equation}
for all $\phi\in C_c^\infty\big((0,T)\times D\big)$ and all $i=1,\dots,k$, and
\begin{equation}\label{eq:MVeuler2}
\begin{split}
\int_0^T\int_D\int_D\ip{\nu^2}{v^i_1v^j_2}\partial_t\phi + \ip{\nu^2}{v^i_1v^j_2v_1}\cdot \nabla_x\phi + \ip{\nu^2}{v^i_1v^j_2v_2}\cdot \nabla_y \phi \\
+ \ip{\nu^2}{v^j_2p_1}\partial_{x^i}\phi + \ip{\nu^2}{v^i_1p_2}\partial_{y^j} \phi\,dxdydt  = 0
\end{split}
\end{equation} 
for all $\phi\in C_c^\infty\big((0,T)\times D^2\big)$ and all $i,j=1,\dots,d$. These are direct analogues of \eqref{eq:weakdef} and \eqref{eq:weakeuler2ndmoment}, respectively, and are the only instances of \eqref{eq:MVeulerk} which will be used in the remainder.
\end{remark}

\begin{remark}
By \eqref{eq:L3}, all integrals in \eqref{eq:aveuler} and \eqref{eq:MVeuler2} are well-defined. 
\end{remark}

\subsection{Energy conservation for statistical solutions}\label{sec:onsagerstatsoln}
We are now ready to prove the ``energy conservation'' part of Onsager's conjecture for statistical solutions. In the same vein as Duchon and Robert \cite{DR00}, we will prove a somewhat stronger result by quantifying the precise energy dissipation $\mathcal{E}(u)$, and prescribing a sufficient condition that ensures that $\mathcal{E}(u)\equiv0$.
\begin{theorem}\label{thm:statsolnenergycons}
Let $\cm$ be a statistical solution of the incompressible Euler equations on $[0,T]\times D$, where either $D=\mathbb{T}^d$ or $D = \R^d$. Let $\rho_\eps(z)=\eps^{-d}\rho(z/\eps)$ be a rotationally symmetric mollifier. Then the distribution $\mathcal{E}(\cm)\in\mathcal{D}'(D)$ given by
\[
\mathcal{E}(\cm)(\psi) := -\frac{1}{2}\lim_{\eps\to0} \int_D\int_{B_\eps(x)} \psi(x)\nabla\rho_\eps(z)\cdot\ip{\nu^2_{x,x-z}}{(v_1-v_2)|v_1-v_2|^2}\,dzdx
\]
is well-defined and independent of the choice of $\rho$, and $\cm$ satisfies
\begin{equation}\label{eq:statsolnenergycons}
\partial_t\ip{\nu^1_{t,x}}{|v|^2} + \sum_{k=1}^d\partial_{x^k}\ip{\nu^1_{t,x}}{|v|^2v^k+2v^kp}=\mathcal{E}(\cm)
\end{equation}
in the sense of distributions. If $\cm$ satisfies the regularity condition
\begin{equation}\label{eq:nubesov}
\liminf_{\eps\to0}\frac{1}{\eps}\int_0^T\int_D\intavg_{B_\eps(x)} \ip{\nu^2_{t,x,y}}{|v_1-v_2|^3}\,dydxdt = 0
\end{equation}
then $\mathcal{E}(\cm)\equiv0$.
\end{theorem}
\begin{remark}
The left-hand side of \eqref{eq:nubesov} equals $\frac{1}{\eps}d_\eps^3(\nu^2)^3$, cf.~\eqref{eq:defd_epstime}. Whereas the requirement of diagonal continuity merely requires that $d_\eps^3(\nu^2)$ vanishes as $\eps\to0$, the regularity assumption \eqref{eq:nubesov} imposes a rate at which it vanishes.
\end{remark}
\begin{remark}\label{rem:kolmogorov}
In turbulence theory, the Kolmogorov four-fifths law states that in a homogeneous (but not necessarily isotropic) turbulent flow, the left-hand side of \eqref{eq:statsolnenergycons} equals
\[
\frac{1}{2}\nabla_z\cdot \big\langle|\delta v(x,z)|^2\delta v(x,z)\big\rangle\Big|_{z=0}, \qquad \delta v(x,z) := v(x)-v(x-z),
\]
where the angle brackets denote the expected value over an ensemble of turbulent flows; cf.~\cite[Section 6.2.5]{Frisch}. It is readily seen that $\mathcal{E}(\cm)$ is a distributional version of the above quantity. {(See also \cite[Section 5]{DR00}.)}
\end{remark}
\begin{proof}
\newcommand{\term}{\mathcal{A}}
Set $j=i$ in \eqref{eq:MVeuler2} and sum over $i$:
\begin{equation}\label{eq:MVeuler2ndmoment_2}
\begin{split}
\sum_i\int_0^T\int_D\int_D\ip{\nu^2_{x,y}}{v^i_1v^i_2}\partial_t\phi + \ip{\nu^2_{x,y}}{v^i_1v^i_2v_1}\cdot \nabla_x\phi + \ip{\nu^2_{x,y}}{v^i_1v^i_2v_2}\cdot \nabla_y \phi \\
+ \ip{\nu^2_{x,y}}{v^i_2p_1}\partial_{x^i}\phi + \ip{\nu^2_{x,y}}{v^i_1p_2}\partial_{y^i} \phi\,dxdydt = 0
\end{split}
\end{equation}
(where we suppress the dependence on $t$). Fix a number $\eps>0$. We choose again the test function $\phi(t,x,y) = \rho_\eps(x-y)\psi(t,x)$, where $\rho_\eps(z) = \eps^{-d}\rho(\eps^{-1}z)$ for a nonnegative, rotationally symmetric mollifier $\rho\in C_c^\infty(D)$ with unit mass and support in $B_0(1)$ and $\psi\in C_c^\infty((0,T)\times D)$. Then as before,
\[
\partial_t\phi = \rho_\eps\partial_t\psi, \qquad \nabla_x \phi = \psi\nabla\rho_\eps+\rho_\eps\nabla\psi, \qquad \nabla_y \phi = -\psi\nabla\rho_\eps.
\]
Continuing from \eqref{eq:MVeuler2ndmoment_2}, we now have
\begin{align*}
\sum_i\int_0^T\int_D\int_D \ip{\nu^2_{x,y}}{v^i_1v^i_2}\rho_\eps\partial_t\psi + \ip{\nu^2_{x,y}}{v^i_1v^i_2(v_1-v_2)}\cdot\nabla\rho_\eps\psi \\
+ \ip{\nu^2_{x,y}}{v^i_1v^i_2v_1}\cdot \rho_\eps\nabla\psi\\
+ \ip{\nu^2_{x,y}}{v_2p_1 - v_1p_2}\cdot \nabla\rho_\eps\psi
+  \ip{\nu^2_{x,y}}{v_2p_1 }\cdot \rho_\eps\nabla\psi\,dxdydt= 0.
\end{align*}
Making the change of variables $z =x-y$ yields
\begin{equation}\label{eq:MVeuler2ndmoment_4}
\begin{split}
\int_0^T \int_D\int_D\partial_t\psi(t,x)\rho_\eps(z)\ip{\nu^2_{x,x-z}}{v_1\cdot v_2}\,dxdzdt \\
+\int_0^T \int_D\int_D\psi(t,x)\nabla\rho_\eps(z)\cdot \ip{\nu^2_{x,x-z}}{(v_1-v_2)(v_1\cdot v_2)}\,dxdzdt \\
+ \int_0^T \int_D\int_D\nabla\psi(t,x)\rho_\eps(z)\cdot \ip{\nu^2_{x,x-z}}{v_1(v_1\cdot v_2)}\,dxdzdt \\
+\int_0^T\int_D\int_D\psi(t,x)\nabla\rho_\eps(z)\cdot \ip{\nu^2_{x,x-z}}{v_2p_1 - v_1p_2}\,dxdzdt \\
+ \int_0^T\int_D\int_D\nabla\psi(t,x)\rho_\eps(z)\cdot \ip{\nu^2_{x,x-z}}{v_2p_1}\,dxdzdt = 0
\end{split}
\end{equation}
Decompose the above into a sum of five terms $\term_1+\dots+\term_5$. We consider each in order.

For $\term_1$ we can write $\term_1 = \term_{1,1}+\term_{1,2}$, where
\begin{align*}
\term_{1,1} &=\frac{1}{2}\int_0^T\int_D\int_D\partial_t\psi(t,x)\rho_\eps(z) \ip{\nu^2_{x,x-z}}{|v_1|^2+|v_2|^2}\,dxdzdt, \\
\term_{1,2} &= -\frac{1}{2}\int_0^T\int_D\int_D\partial_t\psi(t,x)\rho_\eps(z) \ip{\nu^2_{x,x-z}}{|v_1-v_2|^2} \,dxdzdt.
\end{align*}
The first term $\term_{1,1}$ converges as $\eps\to0$ to $\int_0^T\int_D\partial_t\psi(t,x) \ip{\nu^1_{t,x}}{|v|^2}\,dxdt$. Indeed,

\begin{align*}
\term_{1,1} &= \frac{1}{2}\int_0^T\int_D\int_D\partial_t\psi(t,x)\rho_\eps(z)\left( \ip{\nu^2_{x,x-z}}{|v_1|^2} + \ip{\nu^2_{x,x-z}}{|v_2|^2}\right) \,dxdzdt \\
&= \frac{1}{2}\int_0^T\int_D\int_D\partial_t\psi(t,x)\rho_\eps(z)\ip{\nu^1_{x}}{|v|^2}\,dxdzdt \\
&\quad+ \frac{1}{2}\int_0^T\int_D\int_D\partial_t\psi(t,x+z)\rho_\eps(z)\ip{\nu^2_{x+z,x}}{|v_2|^2}\,dxdzdt \\
&= \int_0^T\int_D\int_D\frac{1}{2}\left(\partial_t\psi(x)+\partial_t\psi(x+z)\right)\rho_\eps(z) \ip{\nu^1_{x}}{|v|^2}\,dxdzdt, 
\end{align*}
where we have changed variables $x \mapsto x+z$ in the second equality and used the consistency and symmetry properties of $\nu^2$ in the third equality. Letting $\eps\to0$ we obtain $\term_{1,1}\to\int_0^T\int_D\partial_t\psi(t,x) \ip{\nu^1_{t,x}}{|v|^2}\,dxdt$.

%For the second term $\term_{1,2}$ we use H\"older's inequality, the Besov regularity \eqref{eq:nubesov} and the facts that $\theta_\eps'$ is supported on $[\tau-\eps,\tau]$ and that $|\theta_\eps'(t)| \leq \eps^{-1}\|\theta'\|_{C_b(\R)}$ to obtain
%\begin{align*}
%|\term_{1,2}| &\leq \left(\int_0^T\int_D\int_D \rho_\eps(z)\ip{\nu^2_{x,x-z}}{|v_1-v_2|^3}\,dxdzdt\right)^{2/3}\left(\int_{\tau-\eps}^\tau\int_D\int_D|\theta_\eps'(t)|^3\rho_\eps(z)\,dxdzdt\right)^{1/3} \\
%&\leq C\eps^{2\alpha} \cdot \|\theta'\|_{C_b(\R)}|D|^{1/3}\eps^{-2/3},
%\end{align*}
%which vanishes as $\eps\to0$ since $\alpha>1/3$.\footnote{\textbf{Note to Emil:} I couldn't get the previous argument with Lebesgue Differentiation Theorem (i.e., diagonal continuity) to work here, so I used the Besov regularity again. The argument does work for the term $\term_4$, though. I think it should be possible to only use diagonal continuity; this would be nice because we wouldn't have to assume that $D$ is bounded anymore. How do you show that $\int_0^T \theta_\eps'(t) d_\eps(\nu^2_t)\,dt \to 0$ as $\eps\to0$ when we only have pointwise convergence $d_\eps(\nu^2_t)\to0$?} 
The second term $\term_{1,2}$ converges to zero as $\epsilon\to 0$ by diagonal continuity. Thus,
\[
\lim_{\eps\to0} \term_{1}=\int_0^T\int_D\partial_t\psi(t,x) \ip{\nu^1_{t,x}}{|v|^2}\,dxdt.
\]

For $\term_2$ we can then write $\term_2 = \term_{2,1} + \term_{2,2}$, where
\begin{align*}
\term_{2,1} &= -\frac{1}{2}\int_0^T \int_D\int_D\psi(t,x)\nabla\rho_\eps(z)\cdot \ip{\nu^2_{x,x-z}}{(v_1-v_2)|v_1-v_2|^2} \,dxdzdt, \\
\term_{2,2} &= \frac{1}{2}\int_0^T \int_D\int_D\psi(t,x)\nabla\rho_\eps(z)\cdot\ip{\nu^2_{x,x-z}}{(v_1-v_2)\big(|v_1|^2+|v_2|^2\big)}\,dxdzdt.
\end{align*}
{In the limit $\eps\to0$, the first term is precisely $\mathcal{E}(\cm)(\psi)$. (This limit is well-defined because all the remaining terms converge as $\eps\to0$.) Under the regularity assumption \eqref{eq:nubesov} this term} can be bounded as
\begin{align*}
|\term_{2,1}| &\leq \frac{\|\psi\|_\infty}{2}\int_0^T\int_D\int_D |\nabla\rho_\eps(z)| \ip{\nu^2_{x,x-z}}{|v_1-v_2|^3} \, dxdzdt\\
%&= \frac{\|\psi\|_\infty}{4}\int_0^T \int_D\int_D |\nabla\rho_\eps(z')|\ip{\nu^2_{z,z-2z'}}{|v_1-v_2|^3} \, dz'dzdt \\
&\leq \frac{C\|\psi\|_\infty}{\eps}\int_0^T \int_D\intavg_{B_\eps(0)} \ip{\nu^2_{x,x-z}}{|v_1-v_2|^3} \, dzdxdt \\
&\to 0,
\end{align*}
as $\eps\to0$ (after choosing a suitable subsequence if necessary), where we also used the fact that $\|\nabla\rho_\eps\|_{L^\infty(\R^d)} \leq C\eps^{-1-d}$.

For the second term $\term_{2,2}$ we use the transformation $x\mapsto x+z$ in the term with $|v_2|^2$ and then the transformation $z\mapsto -z$ and thus get
\begin{align*}
\term_{2,2}=\frac{1}{2}\int_0^T \int_D\int_D\nabla\rho_\eps(z)\left(\psi(x)+\psi(x-z)\right)\cdot\ip{\nu^2_{x,x-z}}{|v_1|^2(v_1-v_2)}\,dxdzdt.
\end{align*} 
But now the divergence constraint \eqref{eq:divfree} implies
\begin{align*}
\term_{2,2}=\frac{1}{2}\int_0^T \int_D\int_D\rho_\eps(z)\nabla\psi(x-z)\cdot\ip{\nu^2_{x,x-z}}{|v_1|^2(v_1-v_2)}\,dxdzdt,
\end{align*}
and then the diagonal continuity condition implies $\lim_{\eps\to 0}\term_{2,2}=0$.

Using an argument similar to the treatment of $\term_1$, it is not hard to see that diagonal continuity implies 
\begin{equation*}
\term_3\to \int_0^T\int_D\ip{\nu^1_{t,x}}{|v|^2v}\cdot\nabla\psi(t,x)dxdt
\end{equation*}
as $\eps\to0$, and similarly
\begin{equation*}
\term_5\to \int_0^T\int_D\ip{\nu^1_{t,x}}{pv}\cdot\nabla\psi(t,x)dxdt.
\end{equation*}

Finally, for $\term_4$ we translate $x\mapsto x+z$ in the second term and obtain
\begin{equation*}
\term_4=\int_0^T\int_D\int_D\left[\psi(x)\nabla\rho_\eps(z)\cdot\ip{\nu^2_{x,x-z}}{p_1v_2}-\psi(x+z)\nabla\rho_\eps(z)\cdot\ip{\nu^2_{x,x+z}}{p_1v_2}\right]dxdzdt.
\end{equation*} 
Owing to the divergence condition, the first term is zero, whereas the second term (again invoking the divergence constraint) equals
\begin{equation*}
\term_4=\int_0^T\int_D\int_D\nabla\psi(x+z)\rho_\eps(z)\cdot\ip{\nu^2_{x,x+z}}{p_1v_2}dxdzdt.
\end{equation*} 
Once more invoking diagonal continuity yields
\begin{equation*}
\lim_{\eps\to0}\term_4=\int_0^T\int_D\nabla\psi(x)\cdot\ip{\nu^1_{x}}{pv}dxdt.
\end{equation*} 
Collecting all terms now gives the desired result.

\end{proof}

\section{Probabilistic versus deterministic regularity}\label{sec:regularity}

In this section we will compare the regularity of functions and of correlation measures, and we will show that if a correlation measure is concentrated on a family of $L^2$ functions (soon to be made precise), then this family is at least as regular as the correlation measure (also to be made precise). We will use the notation
\[
d_\eps^q(v) := \left(\int_D\intavg_{B_\eps(x)} \big|v(x)-v(y)\big|^q\,dydx\right)^{1/q}
\]
for a function $v: D \to \R^d$, and for a correlation measure $\cm = (\nu^1,\nu^2,\dots)$ we write
\[
d_\eps^q(\nu^2) := \left(\int_D\intavg_{B_\eps(x)} \ip{\nu^2_{x,y}}{\big|v_1-v_2\big|^q}\,dydx\right)^{1/q}.
\]
(For notational convenience we only look at space-dependent functions in this section.)

In \cite{FLM17} the authors proved that the set of correlation measures, as defined in Definition \ref{def:corrmeas}, are equivalent to the set $\Prob(L^2(D))$ of probability measures on $L^2(D)$. We make this duality more precise in the following definition:

\begin{definition}
A probability measure $\mu \in \Prob(L^2(D;\phase))$ is said to be \define{dual} to a correlation measure $\cm$ from $D$ to $\phase$ provided
\begin{equation}
\int_{D^k} \ip{\nu^k_{x}}{g(x,\cdot)}\, dx = \int_{L^2}\int_{D^k} g(x,v(x_1),\dots,v(x_k))\,dxd\mu(v)
\end{equation}
for every $k\in\N$ and for every Caratheodory function $g : D^k \to C(\phase^k)$.
\end{definition}

\begin{proposition}
Let $\mu \in \Prob(L^2(D;\phase))$ be dual to a space-time correlation measure $\cm$, and assume that $\cm$ is $\alpha$-Besov regular in the sense that
\begin{equation}\label{eq:nubesovthm}
\liminf_{\eps\to0}\frac{d_\eps^q(\nu^2)}{\eps^\alpha} = 0.
\end{equation}
Then
\begin{equation}\label{eq:vbesovthm}
\liminf_{\eps\to0}\frac{d_\eps^q(v)}{\eps^\alpha} = 0
\end{equation}
for $\mu$-almost every $v\in L^2(D;\phase)$. Conversely, if there is a common subsequence $\eps_n\to0$ such that 
\begin{equation}\label{eq:vbesovthm2}
\lim_{n\to\infty}\frac{d_{\eps_n}^q(v)}{\eps_n^\alpha} = 0 \qquad \text{boundedly }
\end{equation}
for $\mu$-almost every $v\in L^2(D;\phase)$, then \eqref{eq:nubesovthm} holds. 
\end{proposition}
\begin{proof}
Using the duality between $\mu$ and $\nu$ and Fatou's lemma yields
\begin{align*}
0 &= \liminf_{\eps\to0} \frac{1}{\eps^{\alpha q}} \int_D\intavg_{B_{\eps}(x)} \ip{\nu^2_{x,y}}{\big|v_1-v_2\big|^q}\,dydx \\
&= \liminf_{\eps\to0} \frac{1}{\eps^{\alpha q}} \int_{L^2}\int_D\intavg_{B_{\eps}(x)} |v(x)-v(y)\big|^q \,dydxd\mu(u) \\
&\geq \int_{L^2}\liminf_{\eps\to0} \frac{1}{\eps^{\alpha q}}\int_D\intavg_{B_{\eps}(x)} |v(x)-v(y)\big|^q \,dydxd\mu(u) \\
&= \int_{L^2}\liminf_{\eps\to0} \frac{d_{\eps}^q(u)^q}{\eps^{\alpha q}} \,d\mu(u).
\end{align*}
The conclusion follows immediately. Conversely, if \eqref{eq:vbesovthm2} holds then the dominated convergence theorem implies that \eqref{eq:nubesovthm} holds along the prescribed subsequence $\eps_n$.
\end{proof}

\section*{Acknowledgments}
U.S.F.\ was supported in part by the grant \textit{Waves and Nonlinear Phenomena} (WaNP) from the Research Council of Norway.

%\bibliographystyle{plain}
%\bibliography{bibliography}

\begin{thebibliography}{10}

\bibitem{BDSV17}
T.~Buckmaster, C.~{De Lellis}, L.~{Sz\'ekelyhidi Jr.}, and V.~Vicol.
\newblock {Onsager's conjecture for admissible weak solutions}.
\newblock {\em Preprint}, 2017.
\newblock arXiv:1701.08678.

\bibitem{CCFS08}
A.~Cheskidov, P.~Constantin, S.~Friedlander, and R.~Shvydkoy.
\newblock {Energy conservation and {O}nsager's conjecture for the {E}uler
  equations}.
\newblock {\em Nonlinearity}, 21(6):1233--1252, 2008.

\bibitem{CET94}
P.~Constantin, W.~E, and E.~Titi.
\newblock {Onsager's conjecture on the energy conservation for solutions of
  Euler's equation}.
\newblock {\em Communications in Mathematical Physics}, 165(1):207--209, 1994.

\bibitem{DiP85}
R.~J. DiPerna.
\newblock {Measure-valued solutions to conservation laws}.
\newblock {\em Archive for Rational Mechanics and Analysis}, 88:223--270, 1985.

\bibitem{DM87}
R.~J. DiPerna and A.~J. Majda.
\newblock {Oscillations and concentrations in weak solutions of the
  incompressible fluid equations}.
\newblock {\em Communications in Mathematical Physics}, 108(4):667--689.

\bibitem{DE17}
T.~Drivas and G.~Eyink.
\newblock {An {O}nsager singularity theorem for turbulent solutions of
  compressible {E}uler equations}.
\newblock {\em Preprint}, 2017.
\newblock arXiv:1704.03409.

\bibitem{DR00}
J.~Duchon and R.~Robert.
\newblock {Inertial energy dissipation for weak solutions of incompressible
  Euler and Navier-Stokes equations}.
\newblock {\em Nonlinearity}, 13(1):249--255, 2000.

\bibitem{Eyi94}
G.~L. Eyink.
\newblock {Energy dissipation without viscosity in ideal hydrodynamics I.
  Fourier analysis and local energy transfer}.
\newblock {\em Physica D: Nonlinear Phenomena}, 78(3):222--240, 1994.

\bibitem{EyiSre06}
G.~L. Eyink and K.~R. Sreenivasan.
\newblock {Onsager and the theory of hydrodynamic turbulence}.
\newblock {\em Rev. Mod. Phys.}, 78:87--135, Jan 2006.

\bibitem{FGSW17}
Eduard Feireisl, Piotr Gwiazda, Agnieszka \'Swierczewska-Gwiazda, and Emil
  Wiedemann.
\newblock {Regularity and energy conservation for the compressible {E}uler
  equations}.
\newblock {\em Arch. Ration. Mech. Anal.}, 223(3):1375--1395, 2017.

\bibitem{FLM17}
U.~S. Fjordholm, S.~Lanthaler, and S.~Mishra.
\newblock {Statistical solutions of hyperbolic conservation laws I:
  Foundations}.
\newblock {\em Accepted for publication in Arch. Ration. Mech. Anal.}, 2017.
\newblock arXiv:1605.05960.

\bibitem{FMW17}
U.~S. Fjordholm, S.~Mishra, and F.~Weber.
\newblock {Statistical solutions of the incompressible Euler equations and
  Kolmogorov's theory of turbulence}.
\newblock {\em In preparation}, 2017.

\bibitem{Frisch}
U.~Frisch.
\newblock {\em {Turbulence}}.
\newblock Cambridge University Press, 1995.

\bibitem{Is17}
P.~Isett.
\newblock {A proof of {O}nsager's conjecture}.
\newblock {\em Preprint}, 2016.
\newblock arXiv:1608.08301.

\bibitem{Kr70}
S.~N. {Kru\v zkov}.
\newblock {First order quasilinear equations with several independent
  variables}.
\newblock {\em Mat. Sb. (N.S.)}, 81 (123):228--255, 1970.

\bibitem{LS16}
T.~M. Leslie and R.~Shvydkoy.
\newblock {The energy balance relation for weak solutions of the
  density-dependent {N}avier-{S}tokes equations}.
\newblock {\em J. Differential Equations}, 261(6):3719--3733, 2016.

\bibitem{Ons49}
Lars Onsager.
\newblock {Statistical hydrodynamics}.
\newblock In {\em {Nuovo Cimento, Suppl.}}, volume~6, pages 279--287, 1949.

\bibitem{RRS16}
J.~C. Robinson, J.~L. Rodrigo, and J.~W.~D. Skipper.
\newblock {A simple integral condition for energy conservation in the 3{D}
  {E}uler equations}.
\newblock {\em Preprint}, 2016.

\bibitem{SW12}
L.~{Sz\'ekelyhidi Jr.} and E.~Wiedemann.
\newblock {Young Measures Generated by Ideal Incompressible Fluid Flows}.
\newblock {\em Arch. Ration. Mech. Anal.}, 206(1):333--366.

\end{thebibliography}

\end{document}